\documentclass[10pt, a4paper]{article}
\usepackage[utf8]{inputenc}
\usepackage{amsfonts}
\usepackage{amsmath}
\usepackage{amsthm}
\usepackage{amssymb}
\usepackage{xcolor}
\usepackage[colorlinks]{hyperref} 
\hypersetup{citecolor=blue}
\usepackage{cleveref}
\usepackage{natbib}
\usepackage{xspace}
\usepackage{comment}

\newcommand{\lsa}{\emph{lsa}\xspace}
\newcommand{\R}{\mathbb R}
\newcommand{\usa}{\emph{usa}\xspace}
\newcommand{\X}{\mathcal X}
\newcommand{\Y}{\mathcal Y}

\newtheorem{lemma}{Lemma}

\newtheorem{theorem}{Theorem}
\newtheorem{corollary}{Corollary}

\newcommand{\reptile}{}
\newtheorem*{repinner}{\reptile}
\newenvironment{refthm}[1]{%
  \renewcommand{\reptile}{Theorem \ref{#1}}%
  \begin{repinner}
}{%
  \end{repinner}
}
\newenvironment{reflem}[1]{%
  \renewcommand{\reptile}{Lemma \ref{#1}}%
  \begin{repinner}
}{%
  \end{repinner}
}

\author{Eugenio Clerico\footnote{Correspondence to eugenio.clerico@gmail.com}}
\date{\small Department of Statistics, University of Oxford}
\title{Borel selection of dominating hyperplanes} 

\begin{document}
\maketitle

\begin{abstract}
\noindent We study a natural measurable selection problem for which the standard uniformisation theorems do not seem to apply directly, yet a Borel selector exists. More precisely, we consider families of finite dimensional functions that admit pointwise domination by affine functionals and ask whether such dominating functionals can be chosen in a Borel measurable way. We prove that this is indeed possible under semi-analytic regularity assumptions. The proof combines a one-dimensional Borel insertion result between an upper and a lower semi-analytic function, derived from Lusin's separation theorem, with an induction on the dimension. As an application, we obtain Borel measurable selections of subgradients for parameter-dependent finite-dimensional convex functions, outside the scope of the standard normal integral framework.
\end{abstract}

\section{Introduction}
A classical problem in descriptive set theory is studying the existence of regular, ideally Borel measurable, selectors for set-valued mappings. Although the classical machinery provides powerful tools for obtaining universally measurable selectors from analytic relations ($\mathbf{\Sigma}^1_1$), the situation becomes remarkably more challenging when analyticity is lost, or when one insists on Borel selectors. The purpose of this paper is to study a natural finite-dimensional  selection problem in which the underlying relation is inherently coanalytic ($\mathbf{\Pi}^1_1$), yet a Borel measurable selector can  be obtained. 

Let $\X$ be a standard Borel space, $\Y\subseteq \mathbb R^n$ a Borel set, and $f:\X\times \Y\to\R$ upper semi-analytic. For each $x\in \X$, consider the set of affine functionals (i.e., hyperplanes) dominating the section $y\mapsto f(x,y)$:
\begin{equation}\label{eq:S(x)}S(x)=\Bigl\{(b,c)\in\R^n\times\R\,:\, f(x,y)\leq b\cdot y+c\,,\; \forall y\in \Y\Bigr\}\,.\end{equation}
A natural question is whether we can associate to each $x$ a dominating hyperplane in a Borel measurable way. Although the coanalytic nature of the graph of the set-valued function $S$ prevents a direct application of much of the classical selection theory, we show that the finite-dimensional linear structure of the problem is strong enough to yield a Borel  selector.

\begin{refthm}{thm:main}[Hyperplane selection]
Let $\X$ be a standard Borel space and $\Y\subseteq\R^n$ a Borel set. Let $f:\X\times \Y\to\R$ be upper semi-analytic. Assume that there exist  $b:\X\to\R^n$ and $c:\X\to\R$ such that $f(x,y)\leq b(x)\cdot y  +c(x)$, for every $x\in\X$ and $y\in \Y$. Then, there exist Borel measurable maps $B:\X\to\R^n$ and $C:\X\to\R$ such that
    $f(x,y)\leq B(x)\cdot y  +C(x)$, for every $x\in\X$ and $y\in\Y$.\\
Equivalently, if $S:\X\mapsto 2^{\R^{n+1}}$, defined in \eqref{eq:S(x)}, is non-empty valued, then it admits a Borel selector.
\end{refthm}

We remark that a natural complement to this result arises when the affine bound is linear (i.e., $c \equiv 0$). In such case, we show that the corresponding Borel affine selector can be chosen to be linear (\Cref{cor:linfunc}).

Although at  first glance \Cref{thm:main} might look like a ``routine'' measurable selection problem, the graph of $S$ (defined through a \emph{universal} quantifier, `$\forall$') is naturally coanalytic, yet our goal is to obtain a strictly Borel selector. This complexity seems to prevent the application of most of the classical selection and uniformisation machinery. First, the relation is not Borel, ruling out standard Borel unicornisation theorems, which typically require additional topological assumptions on the graph's sections such as countability, $\sigma$-compactness, or nonmeagerness (see the introduction of \cite{kechris2025invariant}). Second, a coanalytic relation falls outside the main analytic literature stream, precluding tools built upon the celebrated Jankov-von Neumann theorem for analytic relations (see \cite{srivastava1998course}, Theorem 5.5.2). Third, even within the specialised coanalytic framework, available results fall short: the Novikov-Kond\^o theorem (Theorem 5.14 in \cite{srivastava1998course}) yields only a coanalytic uniformisation, while the Borel selection result of Cenzer-Mauldin (Corollary 3.3 in \cite{burgess1981careful}) strictly requires nonmeager sections. In our setting, this requirement easily fails (e.g., if $f(x,y)=\alpha\cdot y$ for some $\alpha\in\R^n$, then $S(x) = \{\alpha\}\times [0,\infty)$, which is clearly meager in $\R^{n+1}$). Finally, attempting to use the Kuratowski-Ryll-Nardzewski theorem (see, e.g., Theorem 4.1 in \cite{wagner1977survey}) or the Borel selection result from \cite{burgess1981careful} (Theorem 3.1 therein) would require an Effros measurability assumption\footnote{That is, one requires that $\{x\in\X:S(x)\cap U\neq\varnothing\}$ be Borel for every open set $U\subseteq\R^{n+1}$. This is also often called weak Borel measurability of the multifunction.} on the multifunction, which fails here: if $U\subseteq\mathbb R^{n+1}$ is open, then $\{x\in\X\,:\,S(x)\cap U\neq\varnothing\}$ is the projection of a coanalytic relation, and therefore need not be Borel in general.

Rather than relying on classical uniformisation, the proof of Theorem \ref{thm:main} takes a different route, leveraging the following one-dimensional Borel insertion result.
\begin{reflem}{lemma:sandwich}[Analytic sandwich]
Let $\X$ be a standard Borel space, let $u:\X\to\R$ be upper semi-analytic and $l:\X\to\R$ be lower semi-analytic. If $u\leq l$, pointwise on $\X$, then there exists a Borel function $f:\X\to\R$ such that $u\leq f\leq l$, pointwise on $\X$.
\end{reflem}
Despite following easily from Lusin's separation theorem,  Lemma \ref{lemma:sandwich} does not seem to be widely known in this explicit form. To the author's knowledge, it appears only in \cite{clerico2025optimality}, although in a less general formulation, as a technical result in the appendix (Proposition 4). Notably, its conclusion is not a straight consequence of the standard Borel selector machinery, for essentially the same reasons that this machinery does not directly imply Theorem \ref{thm:main}. 

The relevance of the measurable selection problem considered here extends beyond descriptive set theory. For instance, Theorem \ref{thm:main} yields Borel measurable selections of subgradients for suitable parameter-dependent finite dimensional convex functions, while in the  convex-analytic literature, conclusions of this kind are often obtained under stronger regularity assumptions, such as normal-integrand type hypotheses or direct measurability properties of the associated multifunctions (\cite{rockafellar1969measurable,hess1995measurability}). Related domination and selection problems also appear in mathematical finance in connection with superhedging constructions (\cite{bouchard2015arbitrage}), and sequential hypothesis testing when characterising minimal complete classes of tests (\cite{clerico2025optimality}).
\section{Preliminaries}

Throughout this paper, we assume that $\X$ is a standard Borel space, namely a measurable space isomorphic to a Borel subset of a Polish space equipped with its relative Borel $\sigma$-algebra. 

Given a measurable space $\Y$, a \emph{multifunction} (or set-valued map) $S : \X \to 2^\Y$ is a mapping that assigns to each $x \in \X$ a subset $S(x) \subseteq \Y$. The \emph{graph} of $S$ is the relation defined by $\operatorname{Gr}(S) = \{(x, y) \in \X \times \Y : y \in S(x)\}$. If $S$ is non-empty valued (meaning that $S(x) \neq \varnothing$ for every $x \in \X$), we define a \emph{selector} for $S$ as a function $s : \X \to \Y$ such that $s(x) \in S(x)$ for each $x \in \X$. The problem of finding a measurable selector for $S$ is closely related to the \emph{uniformisation} problem for relations. Given a relation $R \subseteq \X \times \Y$, a uniformisation of $R$ is a subset $r \subseteq R$ with single-valued sections. This means that $r$ is the graph of a function $f : \operatorname{proj}_{\X}(R) \to \Y$ such that $(x, f(x)) \in R$ for every $x \in \operatorname{proj}_{\X}(R)$, the projection of $R$ onto $\X$. 

A subset $A \subseteq \X$ is called \emph{analytic}, denoted $A \in \mathbf{\Sigma}^1_1$, if there exist a standard Borel space $\mathcal{Z}$ and a Borel set $B \subseteq \X \times \mathcal{Z}$ such that $A$ is the projection of $B$ onto $\X$. A subset $C \subseteq \X$ is \emph{coanalytic}, denoted $C \in \mathbf{\Pi}^1_1$, if its complement $\X \setminus C$ is analytic. A fundamental result, known as Souslin's theorem, states that a set is Borel if, and only if, it is both analytic and coanalytic. This equivalence is a direct consequence of Lusin's separation theorem (see, e.g., Theorem 14.7 in \cite{kechris1995classical}), which states that if an analytic set $A \subseteq \X$ is contained in a coanalytic set $C \subseteq \X$, then there exists a Borel set $B \subseteq \X$ such that $A \subseteq B \subseteq C$.

These set properties allow us to introduce specific function regularity classes. Let $I\subseteq[-\infty,\infty]$. A function $f : \X \to I$ is defined as \emph{upper semi-analytic} (\usa) whenever the set $\{x \in \X : f(x) > \gamma\}$ is analytic for every $\gamma \in \mathbb{R}$. By symmetry, $f$ is \emph{lower semi-analytic} (\lsa) if $\{x \in \X : f(x) < \gamma\}$ is analytic for all $\gamma \in \mathbb{R}$. Since  $\mathbf{\Sigma}^1_1$ is closed under countable unions and intersections, one can equivalently rewrite the above definitions using non-strict inequalities. Clearly, Borel functions are both \usa and \lsa, and the converse is also true: if a function is simultaneously \usa and \lsa, then it is Borel. A useful property for these semi-analytic regularity classes is their stability  under optimisation over analytic domains. Let $\X$ and $\mathcal{Z}$ be standard Borel spaces, and $A \subseteq \mathcal{Z}$ an analytic set. If $u: \X \times \mathcal{Z} \to I$ is \usa, then $x\mapsto \sup_{z \in A} u(x, z)$ is \usa. Conversely, if $l: \X \times \mathcal{Z} \to I$ is \lsa, then $x\mapsto \inf_{z \in A} l(x, z)$ is \lsa.

For further details about measurable selection and uniformisation, see Chapter 5 of \cite{srivastava1998course}. Background on the basic properties of analytic and coanalytic sets can be found in Chapter 14 of \cite{kechris1995classical}, while Section 7.7 of \cite{bertsekas1978stochastic} gives an introduction to semi-analytic functions.

\section{Analytic sandwich lemma}
Souslin's theorem, which states that a set is Borel if, and only if, it is both analytic and coanalytic, can be seen as a direct consequence of Lusin's separation theorem. In a completely analogous way, the fact that a function is Borel if, and only if, it is both \usa and \lsa may be seen as a special case of the following interpolation result, which extends Proposition 4 from \cite{clerico2025optimality}. Indeed, applying the analytic sandwich lemma below with $u=l=f$ shows that any function which is simultaneously \usa\ and \lsa\ must be Borel.

\begin{lemma}[Analytic sandwich]\label{lemma:sandwich}
Let $\X$ be a standard Borel space, let $u:\X\to\R$ be \usa and $l:\X\to\R$ be \lsa. If $u\leq l$, pointwise on $\X$, then there exists a Borel function $f:\X\to\R$ such that $u\leq f\leq l$, pointwise on $\X$.
\end{lemma}

\begin{proof}
First, note that if $u$ and $l$ are indicator functions, then the result follows directly from Lusin's separation theorem. Indeed, say that $u = \mathbf{1}_A$ and $l = \mathbf{1}_C$. Since $u$ is \usa,  $A = \{u>1/2\}$ is analytic. On the other hand, $C = \{ l > 1/2 \}$ is coanalytic as $l$ is \lsa. Asking that $u\leq l$ pointwise implies that $A\subseteq C$. By Lusin's separation theorem, there is a Borel $B$ such that $A\subseteq B\subseteq C$. Therefore, $f = \mathbf{1}_B$ is Borel measurable and satisfies $u\leq f\leq l$.

Next, let us consider the case of simple functions, namely we assume that $u$ and $l$ are valued on a finite set $\mathcal Y = \{y_1,\dots,y_N\}$. We assume that $\mathcal Y$ is ordered increasingly, namely $y_{i+1}>y_i$ for all $i$. Let $A_i = \{u\geq y_i\}$, which is analytic, and $C_i = \{l\geq y_i\}$, which is coanalytic. Because $u\leq l$ pointwise, $A_i\subseteq C_i$. By Lusin's separation theorem there is a Borel set $B_i$ such that $A_i\subseteq B_i\subseteq C_i$. For $k=1,\dots N$, let $B'_k = \bigcup_{i=k}^N B_i$. Note that we have $A_k \subseteq B_k\subseteq B'_k\subseteq  C_k$, since $B_i\subseteq C_i\subseteq C_k$ for every $i\geq k$. Now, let $D_N = B'_N$, and then define iteratively $D_i = B'_i\setminus B'_{i+1}$ (note that by construction $B'_{i+1}\subseteq B'_i$ for all $i$ and $B_1=\X$). The sets $D_i$ are Borel, so $f = \sum_{i=1}^N y_i\mathbf{1}_{D_i}$ is Borel measurable. By construction, we also have that $\{f\geq y_i\}=B'_i$, which shows that $u\leq f\leq l$ pointwise.

Now, we can focus on the case where both $u$ and $l$ are bounded functions. By rescaling and translating if necessary, we can assume without loss of generality that they both take values in $[0,1]$. First, let us show that there is a non-decreasing sequence $(u_n)_{n\geq 1}$ of \usa simple functions that converge pointwise to $u$. This can be easily constructed. For each $n$, for $t=0,\dots,2^n$, let $A_t^n = \{u\geq t/2^n\}$. Now define $u_n$ as follows. On $A_{2^n}^n$ we set $u_n$ equal to $1$. For $t = 0,\dots,2^n-1$, for $x\in A_t^n\setminus A_{t+1}^n$, set $u_n(x) = t/2^n$. Since $A_0^n = \mathcal X$ and each $A_t^n\supseteq A_{t+1}^n$, $u_n$ is well defined on the whole $\mathcal X$. It is a quick check that each $u_n$ is simple and \usa, and that the sequence $(u_n)_{n\geq 1}$ is non-decreasing and converges pointwise to $u$. With an analogous construction, we can obtain a non-increasing sequence $(l_n)_{n\geq1}$ of \lsa simple functions, converging pointwise to $l$. From what we derived earlier, for each $n$ there is a Borel function $f_n$, such that $u_n\leq f_n\leq l_n$ pointwise. Letting $f=\limsup_{n\to\infty} f_n$, we get that $u\leq f\leq l$ pointwise, and $f$ is Borel measurable.

Finally, let us consider the general case, where $u$ and $l$ are potentially unbounded.  For each integer $n\geq 1$, define $u_n = \max\{-n, \min\{u,n\}\}$ and $l_n = \max\{-n, \min\{l,n\}\}$. Then each $u_n$ is bounded and \usa, and each $l_n$ is bounded and \lsa. Moreover, $u_n\leq l_n$ pointwise. By what we have already shown, we know that there exists $f_n$ Borel such that $u_n\leq f_n\leq l_n$. Clearly, $u_n\to u$ and $l_n\to l$ for $n\to\infty$. We can now define $f = \limsup_{n\to\infty} f_n$. This is a Borel function, and satisfies  $u\leq f\leq l$ pointwise.
\end{proof}
We note that the result readily extends, with the very same proof, to $u$ and $l$ taking values in $[-\infty, \infty]$, although in such a case we also need to allow $f$ to take infinite values.

For dealing with the base case ($n=0$) in \Cref{thm:main}, we will  need one further preliminary lemma, which can be derived via a generalisation of Lusin's separation theorem due to Novikov.
\begin{lemma}\label{lemma:novikov}
Let $\X$ be a standard Borel space, and let $u:\X\to\R$ be \usa. Then, there exists a Borel map $f:\X\to\R$ such that $f\geq u$ pointwise on $\X$.
\end{lemma}
\begin{proof}
For each integer $n \geq 1$, we define $A_n = \{x \in \X : u(x) \leq n\}$. Because $u$ is \usa,  $A_n$ is  coanalytic. Furthermore, since $u$ takes values in $\R$, $\bigcup_{n=1}^\infty A_n = \X$. Therefore, Novikov's generalised separation theorem (\cite{kechris1995classical}, Theorem 28.5) implies the existence of a sequence of Borel sets $B_n$ such that $B_n \subseteq A_n$ for all $n$, and $\bigcup_{n=1}^\infty B_n = \X$. We  define $f: \X \to \R$ by $f(x) = \min \{n \geq 1 : x \in B_n\}$.
Since the sets $B_n$ cover $\X$, $f$ is well-defined and finite everywhere. Moreover, $f$ is  Borel measurable and dominates $u$.
\end{proof}

As a side remark, with some more work but following a similar reasoning one can obtain a stronger version of \Cref{lemma:novikov}, which is again in the form of a sandwich lemma: if $u$ \usa is valued in $[-\infty, \infty)$ and $l$ \lsa in $(-\infty, \infty]$, then there exists a real-valued Borel interpolator $f: \X \to \R$, assuming $u \leq l$.

\section{Hyperplane selection}
\begin{theorem}[Hyperplane selection]\label{thm:main}
    Let $\X$ be a standard Borel space and $\Y\subseteq\R^n$ a Borel set. Let $f:\X\times \Y\to\R$ be \usa. Assume that there exist maps $b:\X\to\R^n$ and $c:\X\to\R$ such that $f(x,y)\leq b(x)\cdot y  +c(x)$, for every $x\in\X$ and $y\in \Y$. Then, there exist Borel measurable maps $B:\X\to\R^n$ and $C:\X\to\R$ such that
    $f(x,y)\leq B(x)\cdot y  +C(x)$, for every $x\in\X$ and $y\in\Y$.
\end{theorem}
\begin{proof}
    Without loss of generality, we may assume $\Y = \mathbb R^n$. Indeed, if $\Y \subsetneq \mathbb R^n$, we extend $f$ to $\tilde{f}: \X \times \mathbb R^n \to \R$ by setting $\tilde{f}(x,y) = f(x,y)$ if $y \in \Y$ and $\tilde{f}(x,y) = -\|y\|^2$ (where $\|\cdot\|$ is the Euclidean norm) if $y \notin \Y$. Since $\Y$ is a Borel set, $\tilde{f}$ is \usa. We now verify that $\tilde{f}$ satisfies the affine domination hypothesis. For $y \in \Y$, we have $\tilde{f}(x,y) = f(x,y) \leq b(x)\cdot y + c(x)$ by assumption. On the other hand, for $y \notin \Y$, we use $-\|y\|^2 \leq \|y + b(x)/2\|^2 - \|y\|^2 =  b(x)\cdot y + \|b(x)\|^2/4$. Thus, $\tilde{f}(x,y) \le b(x)\cdot y + \tilde{c}(x)$, for all $x\in\X$ and $y\in\R^n$, where $\tilde{c}(x) = \max(c(x), \|b(x)\|^2/4)$. Therefore, we now proceed assuming $\Y = \mathbb R^n$.

    We prove the result by induction on the dimension. If $n=0$, then $\R^0 = \{0\}$ and the hypothesis  reduces to $f(x,0)\leq c(x)$ for every $x\in\X$. Since $f$ is \usa, the section $x\mapsto f(x,0)$ is also a \usa map from $\X$ to $\R$. By \Cref{lemma:novikov}, there is a real-valued $C:\X\to\R$ such that  $f(x,0)\leq C(x)$ for all $x\in\X$. 

    Now, let $n\geq 1$ and assume that the theorem holds up to $n-1$. We let $f:\X\times \R^n\to\R$  satisfy the affine domination hypothesis with $b$ and $c$.  
    We partition $\R^n$ into $\R^n_> = \{y\in \R^n\,:\,y_n>0\}$, $\R^n_<=\{y\in \R^n\,:\,y_n<0\}$, and $\R^n_0 = \{y\in \R^n\,:\,y_n=0\}$. Define the set 
    $$\Gamma = \left\{(\tilde y, y, y')\in \R^n_0\times \R^n_>\times \R^n_<\,:\,\frac{y_{n} y' - y'_{n} y}{y_{n}-y'_{n}} = \tilde y\right\}\,.$$ This is a Borel set, and so the function $g:\X\times \R^n_0\times \R^n_>\times \R^n_<\to[-\infty,\infty)$, defined as 
    $$g(x,\tilde y, y, y') = \begin{cases}
    \frac{y_{n}f(x,y') - y'_{n}f(x,y)}{y_{n}-y'_{n}} & \text{if $(\tilde y, y, y')\in\Gamma$};\\-\infty & \text{otherwise,}\end{cases}$$
    is \usa (note that $y_{n}>0$ and $y'_{n}<0$). For any $(x, \tilde y, y, y')\in \X\times \R^n_0\times \R^n_>\times \R^n_<$,$$g(x, \tilde y, y, y') \leq \frac{y_{n}  \big(b(x)\cdot y'+c(x)\big) - y'_{n} \big(b(x)\cdot y+c(x)\big)}{y_{n}-y'_{n}} =   b(x) \cdot\tilde y + c(x)\,,$$  as the left side is finite only if $(\tilde y, y,y')\in\Gamma$. Since the upper-bound on $g$ is always finite, we can define $h:\X\times \R^n_0\to[-\infty,\infty)$ as  
    $$h(x,\tilde y) = \sup_{y\in \R^n_>,y'\in \R^n_<} g(x, \tilde y, y, y')\,,$$ 
    which is \usa, as the supremum of a \usa function over an analytic set, and satisfies $h(x,\tilde y)\leq b(x)\cdot \tilde y + c(x)$, for all $(x,\tilde y)\in\X\times \R^n_0$.

    We can now define $F:\X\times \R^n_0\to \R$ as
    $$F(x, \tilde y) = \max\big(h(x,\tilde y),f(x,\tilde y)\big)\,.$$ $F$ is the maximum between two \usa functions, hence it is \usa. Moreover, since both $h$ and $f$ are dominated by $(x,\tilde y)\mapsto b(x)\cdot\tilde y + c(x)$, we have that
    $$F(x,\tilde y) \leq b(x)\cdot \tilde y + c(x) = \sum_{i=1}^{n-1} b_i(x)\tilde y_i + c(x)\,,$$ for every $\tilde y\in \R^n_0$ and $x\in\X$, where we used that $\tilde y_{n} = 0$ as $\tilde y\in \R^n_0$. By the inductive hypothesis, since  $\R^n_0$ can be canonically identified with  $\R^{n-1}$, we have that there exist Borel functions (from $\X$ to $\R$) $B_1,\dots, B_{n-1}$, and $C$, such that
    $$F(x, \tilde y) \leq \sum_{i=1}^{n-1} B_i(x)\tilde y_i + C(x)\,,$$
    for every $x\in\X$ and $\tilde y\in \R^n_0$. 

    Fix $y\in \R^n_>$ and $y'\in \R^n_<$, and let $\tilde y = \frac{y_{n}y' - y_{n}'y}{y_{n}-y'_{n}}$. Clearly, $\tilde y\in \R^n_0$ and $(\tilde y, y,y')\in\Gamma$. In particular,
    $$\frac{y_{n}f(x,y') - y'_{n}f(x,y)}{y_{n}-y'_{n}}= g(x, \tilde y, y, y')\leq F(x,\tilde y) \leq \sum_{i=1}^{n-1}B_i(x)\frac{y_{n}y_i' - y_{n}'y_i}{y_{n}-y'_{n}} + C(x)\,.$$
    Rearranging (and using that $y_{n}-y'_{n}>0$), we get
    $$\frac{f(x,y)-\sum_{i=1}^{n-1} B_i(x)y_i -C(x)}{y_{n}}\leq \frac{f(x,y')-\sum_{i=1}^{n-1} B_i(x)y'_i-C(x)}{y'_{n}}\,.$$
    This holds for every $x\in\X$, $y\in \R^n_>$, and $y'\in \R^n_<$. We now define the functions $U$ and $L$, from $\X$ to $[-\infty,\infty]$, as
    \begin{align*}U(x) &= \sup_{y\in \R^n_>}\frac{f(x,y)-\sum_{i=1}^{n-1} B_i(x)y_i -C(x)}{y_{n}}\,;\\ L(x) &= \inf_{y'\in \R^n_<}\frac{f(x,y')-\sum_{i=1}^{n-1} B_i(x)y'_i-C(x)}{y'_{n}}\,.\end{align*} Since $\R^n_<$ and $\R^n_>$ are non-empty, $L(x)<\infty$ and $U(x)>-\infty$ for any $x\in\X$. As we have $U\leq L$ pointwise, we can view $U$ and $L$ as functions from $\X$ taking values in $\R$. Note that $U$ is \usa and $L$ is \lsa, since they are respectively the supremum and the infimum on analytic domains of \usa and \lsa functions. In particular, the analytic sandwich lemma (\Cref{lemma:sandwich}) applies and we can find a Borel map $B_n:\X\to\R$ such that $U(x)\leq B_n(x)\leq L(x)$, for all $x\in\X$. 
    
    We now define $B:\X\to\R^n$ as the Borel map with components $B_1,\dots, B_n$. We have  obtained that, for any $x\in\X$ and $y\in \R^n_>\cup \R^n_<$, 
    $$f(x,y) \leq B(x)\cdot y + C(x)\,.$$ We are only left to check the case $y\in \R^n_0$. However, in such case we simply have that
    $$f(x,y)\leq F(x,y)\leq \sum_{i=1}^{n-1} B_i(x) y_i +C(x) = B(x)\cdot y + C(x)\,,$$ since $y_{n}=0$.
\end{proof}
It is worth remarking that the Borel assumption on the domain $\Y$ in Theorem \ref{thm:main} can be relaxed to mere analyticity, provided we slightly extend the definition of upper semi-analyticity beyond standard Borel spaces. Indeed, if $\Y \subseteq \mathbb R^n$ is merely an analytic set, the product space $\X \times \Y$ is generally not standard Borel. In this context, a function $f: \X \times \Y \to [-\infty, \infty)$ is said to be \usa if, for every $\gamma \in \mathbb R$, the superlevel set $\{(x,y) \in \X \times \Y : f(x,y) > \gamma\}$ is analytic in the standard Borel space $\X \times \mathbb R^n$. Under this extended definition, our Borel selection result holds for arbitrary analytic domains $\Y$.

\begin{corollary}\label{cor:analytic}
Let $\X$ be a standard Borel space and $\Y\subseteq\R^n$ an analytic set. Let $f:\X\times \Y\to\R$ be \usa in the extended sense described above. Assume there exist maps $b:\X\to\R^n$ and $c:\X\to\R$ such that $f(x,y)\leq b(x)\cdot y + c(x)$ for all $x\in\X$ and $y\in\Y$. Then, there exist Borel measurable maps $B:\X\to\R^n$ and $C:\X\to\R$ satisfying $f(x,y)\leq B(x)\cdot y + C(x)$ for all $x\in\X$ and $y\in\Y$.
\end{corollary}
\begin{proof}We extend $f$ to $\X \times \R^n$ by defining $g: \X \times \R^n \to [-\infty, \infty)$ as $g(x,y) = f(x,y)$ if $y\in\Y$ and $g(x,y)=-\infty$ if $y\notin\Y$. For any $\gamma \in \R$, the superlevel set $\{(x,y) \in \X \times \R^n : g(x,y) > \gamma\}$ is precisely $\{(x,y) \in \X \times \Y : f(x,y) > \gamma\}$. Since $f$ is \usa in the extended sense, this set is analytic in $\X \times \R^n$, and hence $g$ is \usa. To obtain a function valued in $\R$, we define $\tilde{f}: \X \times \R^n \to \R$ as $\tilde{f}(x,y) = \max(g(x,y), -\|y\|^2)$. As the maximum of two \usa functions, $\tilde{f}$ is \usa. Proceeding as in the beginning of the proof of \Cref{thm:main}, one easily shows that $\tilde{f}(x,y) \le b(x)\cdot y + \tilde{c}(x)$ for all $x \in \X$ and $y \in \R^n$, where $\tilde{c}(x) = \max(c(x), \|b(x)\|^2/4)$. Applying \Cref{thm:main} to $\tilde{f}$ yields Borel  maps $B: \X \to \R^n$ and $C: \X \to \R$ such that $\tilde{f}(x,y) \leq B(x)\cdot y + C(x)$, for all $x \in \X$ and $y \in \R^n$. Since $f(x,y) = g(x,y) \leq \tilde{f}(x,y)$ on $\Y$, we conclude.\end{proof}

\section{Linear functional selection}
We can leverage the extended formulation on analytic sets of \Cref{cor:analytic} to deduce a selection theorem for  linear functionals.

\begin{corollary}\label{cor:linfunc}
Let $\X$ be a standard Borel space and $\Y\subseteq \mathbb R^n$ a Borel set. Let $f:\X\times \Y\to\R$ be \usa. Assume there exists a map $a:\X\to\R^n$ such that, for all $x\in\X$ and $y\in\Y$, $f(x,y)\le a(x)\cdot y$. Then, there exists a Borel  map $A:\X\to\mathbb R^n$, satisfying $f(x,y)\leq A(x)\cdot y$ for all $x\in\X$ and $y\in\Y$.
\end{corollary}

\begin{proof}
We define the set
    $$\Gamma = \bigl\{(z, y, t) \in \R^n \times \Y \times (0,\infty) \,:\, z = ty\bigr\}\,.$$
Let $g: \X \times \R^n \times \Y \times (0,\infty) \to [-\infty, \infty)$ be given by
$$g(x, z, y, t) = \begin{cases}
    tf(x,y) & \text{if } (z, y, t) \in \Gamma; \\
    -\infty & \text{otherwise.}
\end{cases}$$
Using that $\Gamma$ is a Borel set, $f$ is \usa, and $t > 0$,  it is easily checked that $g$ is \usa.
Moreover, for any $(x, z, y, t) \in \X \times \R^n \times \Y \times (0,\infty)$, we have $$g(x, z, y, t) \leq ta(x)\cdot y = a(x)\cdot z\,,$$
as the left side is finite only if $(z, y, t) \in \Gamma$, (i.e., for $z = ty$). Since this upper bound is finite, we can define $h: \X \times \R^n \to [-\infty, \infty)$ as
$$h(x,z) = \sup_{(y,t) \in \Y \times (0,\infty)} g(x, z, y, t)\,,$$
which is \usa (as the supremum of a \usa function over an analytic domain) and satisfies $h(x,z) \leq a(x)\cdot z$ for all $(x,z) \in \X \times \R^n$. 

We denote by $\mathcal C$ the projection of $\Gamma$ on $\R^n$, which is also the cone generated by $\Y$. Since $\Gamma$ is Borel, $\mathcal C$ is analytic. We note that for each $z \in \mathcal{C}$ there is at least one pair $(y,t)\in\Y\times(0,\infty)$ such that $(z,y,t)\in\Gamma$, and so $g(x,z,y, t)>-\infty$. In particular, $h$ is real-valued on $\X \times \mathcal{C}$. Moreover, for every $x\in\X$ and $z\in\mathcal C$, we have already shown that $h(x,z)\le a(x)\cdot z$. Therefore, $h$ satisfies the hypotheses of \Cref{cor:analytic}, on the analytic domain $\X\times\mathcal C$. We conclude that there exist Borel measurable maps $B:\X\to\R^n$ and $C:\X\to\R$, such that
$$h(x,z)\le B(x)\cdot z + C(x)\,,$$
for every $x\in\X$ and $z\in\mathcal C$.

We now show that $h$ is positively homogeneous on $\X\times\mathcal C$. Indeed, for any $x\in\X$, $z\in\mathcal C$, and $\lambda>0$, we have
$$h(x,\lambda z) = \sup_{\substack{y\in\Y,\,t>0\\ \lambda z = ty}} tf(x,y)
= \sup_{\substack{y\in\Y,\,t>0\\z = ty}} \big(\lambda tf(x,y)\big) = \lambda h(x,z)\,.$$
In particular, we can fix any $y\in\Y$, and since $\lambda y\in\mathcal C$ we get
$$h(x,y)=h(x,\lambda y)/\lambda\leq B(x)\cdot y + C(x)/\lambda\,.$$
As this must hold for every $\lambda>0$, we can let $\lambda\to\infty$ and we obtain
$$h(x,y)\le B(x)\cdot y\,.$$

Finally, let us fix any $x\in\X$ and $y\in\Y$. Since $(y,y,1)\in\Gamma $, we have that
$$f(x,y) = g(x,y,y,1) \leq h(x,y) \leq B(x)\cdot y\,.$$
Setting $A = B$ concludes the proof.
\end{proof}
As we did for \Cref{cor:analytic}, the assumption that $\Y$ is a Borel set in Corollary \ref{cor:linfunc} can be relaxed to mere analyticity (where upper semi-analyticity is understood in the extended sense as in \Cref{cor:analytic}). The proof remains essentially identical: the generated cone $\mathcal{C}$ is still analytic, and one  invokes the analytic extension (Corollary \ref{cor:analytic}) in place of Theorem \ref{thm:main} to obtain the intermediate affine bound.

We also remark that the linear framework of \Cref{cor:linfunc}   extends to arbitrary Borel measurable transformations of the domain, as shown in the next corollary. Notably, this result recovers \Cref{thm:main} as a special case.
\begin{corollary}\label{cor:feature}
    Let $\X$ and $\Y$ be standard Borel spaces. Let $f:\X\times\Y\to\R$ be \usa and $\phi:\Y\to\R^n$  Borel. Assume there exists a map $a:\X\to\R^n$ such that, for all $x\in\X$ and $y\in\Y$, $f(x,y)\le a(x)\cdot \phi(y)$. Then, there exists a Borel  map $A:\X\to\mathbb R^n$, satisfying $f(x,y)\leq A(x)\cdot \phi(y)$ for all $x\in\X$ and $y\in\Y$.
\end{corollary}
\begin{proof}
     We define $h:\X\times\Y\times\R^n\to[-\infty,\infty)$ as
    $$h(x,y,z) = \begin{cases}f(x,y)&\text{if $z=\phi(y)$;}\\-\infty&\text{otherwise.}\end{cases}$$
    For any $\gamma\in\R$, the set $S = \{(x,y,z) \in \X\times\Y\times\R^n\,:\,h(x,y,z)>\gamma\}$ is analytic. Indeed, it can be written as $S = S_1\cap S_2$, where
    $$S_1 = \{(x,y)\in \X\times\Y\,:\,f(x,y)>\gamma\}\times\R^n$$ is analytic since $f$ is \usa, and 
    $$S_2 = \X\times \{(y,z)\in\Y\times\R^n\,:\,z=\phi(y)\}$$ is Borel, as the graph of a Borel map is a Borel set (\cite{kechris1995classical}, Theorem 14.12). In particular, $h$ is \usa. 
    
    Let $\mathcal Z=\phi(\Y)$, the image of $\phi$, which is an analytic set in $\R^n$ (see, e.g., Proposition 14.4 in \cite{kechris1995classical}). We define $g:\X\times\R^n\to[-\infty, \infty)$ as 
    $$g(x,z) = \sup_{y\in\Y}h(x,y,z)\,,$$
    where we note that $g(x,z)\leq a(x)\cdot z<\infty$ for all $(x,z)\in\X\times\mathcal Z$. $g$ is \usa, as the supremum of an \usa function on an analytic domain. In particular, its restriction to $\X\times\mathcal Z$ is \usa in the extended sense of \Cref{cor:analytic}. Moreover, when restricted to $\X\times\mathcal Z$, $g$ takes finite values. In particular, applying the analytic extension of \Cref{cor:linfunc} to $g$ on the domain $\X\times\mathcal Z$, we obtain a Borel measurable map $A:\X\to\R^n$ such that $g(x,z) \le A(x)\cdot z$ for all $(x,z) \in \X\times\mathcal Z$. We thus conclude that $$f(x,y)\leq g(x,\phi(y))\leq A(x)\cdot\phi(y)\,,$$ for every $x\in\X$ and $y\in\Y$. 
    \end{proof}

\section{Subgradient selection}

A basic theme in convex and variational analysis is the measurable dependence of subdifferentials on parameters. This question is central in stochastic optimisation, optimal control, and mathematical finance, where one often needs to choose supporting hyperplanes or subgradients in a measurable way. The classical framework for such problems is the theory of \emph{normal integrands} (see, e.g., \cite{rockafellar1976integrands} and  \cite{castaing1977convex}). In that setting, one assumes lower semicontinuity of the sections together with suitable epigraphical measurability, which ensures  the propagation of measurability through the basic operations of convex analysis. In particular, \cite{hess1995measurability} proved that, for a proper normal integrand, the conjugate integrand is again normal, and the subdifferential multifunction along a measurable section is Effros measurable. More recent developments in stochastic convex analysis continue to build on this perspective; see, e.g., \cite{kiiski2017optional,pennanen2024convex}, and \cite{bui2024interchange}.

However, natural operations such as projection or marginalisation can lead outside the normal-integrand framework. Even when the original setting is well behaved, the resulting functions need not preserve lower semicontinuity or Borel regularity in a form accessible to the classical theory, and are often more naturally described within the semi-analytic setting.

In finite dimensions, the hyperplane selection theorem yields, through  \Cref{cor:linfunc}, a different approach to measurable subgradient selection. Without focusing on normal-integrand structure or prior measurability of the subdifferential multifunction, one can directly obtain a Borel measurable subgradient selector in semi-analytic settings that fall outside the scope of classical theory, as exemplified by the following corollary.

\begin{corollary}\label{cor:subgradient}
Let $\X$ be a standard Borel space, and let $g:\X\times\R^n\to\R$ be \lsa. Assume that, for every $x\in\X$, the section $y\mapsto g(x,y)$ is convex, $g(x,0)=0$, and $\partial_y g(x,0)\neq\varnothing$, where $\partial_y g(x,0)$ denotes the subdifferential at $0$ of the section $y\mapsto g(x,y)$. Then, there exists a Borel measurable map $p:\X\to\R^n$ such that, for every $x\in\X$,
$$p(x)\in\partial_y g(x,0)\,.$$
\end{corollary}
\begin{proof}
    For all $x\in\X$, the non-emptiness assumption of the subdifferential $\partial_y g(x,0)$ implies the existence of  $p_x \in \R^n$ such that, $g(x,y) \geq g(x,0) + p_x \cdot y = p_x\cdot y$ (since $g(x,0)=0$ by hypothesis), for all $y \in \R^n$. Applying \Cref{cor:linfunc} to $f = -g$, which is \usa, we find a Borel $A:\X\to\R^n$ such that $f(x,y)\leq A(x)\cdot y$ for all $x\in\X$ and $y\in\R^n$. Setting $p=-A$ concludes the proof. 
\end{proof}
We remark that the assumptions that $g(x,0) = 0$ and that the subdifferential is evaluated at the origin can be readily relaxed. If $y_0:\X\to\R^n$ is a Borel map such that $x \mapsto g(x,y_0(x))$ is Borel measurable, the existence of a Borel  selector $p(x) \in \partial_y g(x, y_0(x))$ follows directly by applying Corollary \ref{cor:subgradient} to the shifted function $\tilde{g}(x,y) = g(x, y + y_0(x)) - g(x,y_0(x))$.

\section{Discussion}

We have proved that, in a finite-dimensional setting, it is possible to obtain a Borel measurable selection of affine functionals dominating an upper semi-analytic function, even though the associated relation is generally only coanalytic and the problem lies outside the direct scope of standard selection theorems. The key insight is that finite-dimensional linear structure compensates for the descriptive-set-theoretic complexity of the problem: a one-dimensional Borel insertion argument, combined with an induction on the dimension, is sufficient to recover a Borel selector. This perspective takes inspiration from a one-dimensional argument in \cite{clerico2025optimality}, where an early version of the analytic sandwich lemma was used to establish a result almost equivalent to the case $n=1$ of \Cref{cor:linfunc}, albeit in an implicit and problem-specific form. By abstracting and building upon this core mechanism, the present paper develops a general finite-dimensional affine selection framework. This provides a natural basis for multivariate extensions of the hypothesis testing results from \cite{clerico2025optimality}, with potential relevance to statistics and decision theory more broadly.

In mathematical finance, our result complements the standard quasi-sure approach (\cite{bouchard2015arbitrage,nutz2016utility}). In that literature, the relevant regularity notion is usually universal measurability, as required for dynamic programming, measurable kernel selection, and quasi-sure statements under non-dominated families of priors. The classical framework therefore mixes different regularity levels: one assumes strong regularity on the input side, yet typically obtains selectors that are only universally measurable, which in the words of \cite{boistard2025robust} means that one \textit{``ha[s] to assume a lot...\ to obtain little''}. To address this mismatch, \cite{carassus2024nonconcave} and \cite{boistard2025robust} propose supplementing ZFC with Projective Determinacy and reformulating the theory in projective terms. Theorem \ref{thm:main} suggests a different, structural way to reduce the same mismatch in finite-dimensional linear problems, without altering the underlying set-theoretic axioms. This suggests that a Borel approach may be worth exploring in finite-dimensional problems, such as one-step superhedging or related supporting-hyperplane constructions, where the literature typically settles for universal measurability.

As a final technical remark, it is worth stressing that the argument underlying the proof of Theorem \ref{thm:main} is essentially finite-dimensional. One might be tempted to extend it directly to spaces of countable algebraic dimension by constructing the Borel dominating function on a basis, one component at a time. Such a strategy would require a consistent procedure, allowing one to enlarge the construction without revising the components already chosen. Instead, each step eliminates one coordinate by encoding the corresponding constraints into a new auxiliary problem in lower dimension. The coefficients in the reduced problem must then be chosen again at each stage, rather than fixed once and kept unchanged through later extensions. This approach is effective in  the finite-dimensional setting, where it can be repeated finitely many times. Whether related Borel selection results remain valid in spaces with a countable algebraic basis, or in more general infinite-dimensional vector spaces, is an open question.

\paragraph{Acknowledgments.}
I am grateful to Sebastian Arnold for insightful and extensive discussions that helped inspire this work. I also thank Riccardo Camerlo for the valuable feedback and suggestions to improve this work. I acknowledge the use of Gemini (Google, Gemini 3.1 Pro) and ChatGPT (OpenAI, GPT-5.4 Thinking) during the preparation of this paper, for polishing and improving the exposition, and for exploring and simplifying possible proof ideas. I revised and edited all AI-generated material, and I take full responsibility for the content and correctness of the paper.
\bibliographystyle{abbrvnat}
\bibliography{bib}

\end{document}